\documentclass[12pt,leqno]{amsart}

\usepackage{graphicx}
\usepackage{amsmath,amsthm,amsfonts,amssymb,latexsym}

\usepackage[all]{xy}
\usepackage{eucal}
\usepackage{euler}
\usepackage{times}

\swapnumbers

\newcounter{ictr}
\newenvironment{ilist}{\begin{list}
                         {\textup{(\alph{ictr})}}
                         {\usecounter{ictr}
                          \setlength{\leftmargin}{0.4truein}
                          \setlength{\itemsep}{0.0truein}
                          \setlength{\labelwidth}{0.3truein}}}
                      {\end{list}}

\newcounter{nctr}
\newenvironment{nlist}{\begin{list}
                         {\textup{(\arabic{nctr})}}
                         {\usecounter{nctr}
                          \setlength{\leftmargin}{0.6truein}
                          \setlength{\itemsep}{0.0truein}
                          \setlength{\labelwidth}{0.3truein}}}
                      {\end{list}}

\newtheorem*{thma}{Theorem A}
\newtheorem*{thmb}{Theorem B}

\newtheorem{thm}{Theorem}[section]

\newtheorem{lem}[thm]{Lemma}

\newtheorem{prop}[thm]{Proposition}

\theoremstyle{definition}

\theoremstyle{remark}

\newtheorem*{ex}{Example}
\newtheorem*{rem}{Remark}

\numberwithin{equation}{section}

\newcommand{\h}{\mathcal{H}}

\newcommand{\N}{\mathbb{N}}

\newcommand{\eps}{\varepsilon}

\newcommand{\normi}[1]{\left\|#1\right\|}

\DeclareMathOperator{\Prob}{Prob}
\DeclareMathOperator{\supp}{supp}
\DeclareMathOperator{\CAT}{CAT}

\newcommand{\Hawaii}{Hawai\kern.05em`\kern.05em\relax i}
\newcommand{\Manoa}{M\=anoa}

\title{Complexes and exactness of certain Artin groups.}

\author{Erik Guentner}
\address{University of \Hawaii~at \Manoa, 
Department of Mathematics,
2565 McCarthy Mall, Honolulu, HI 96822}
\email{erik@math.hawaii.edu}

\author{Graham A. Niblo}
\address{School of Mathematics, University of Southampton, Highfield,
  Southampton, SO17 1SH, England} 
\email{G.A.Niblo@soton.ac.uk}

\thanks{The first author was partially supported by NSF grant
  DMS-0349367. The second author was supported by EPSRC grant
  EP/H04874X/1.} 

\begin{document}

\begin{abstract}
  In his work on the Novikov conjecture, Yu introduced Property $A$ as
  a readily verified criterion implying coarse embeddability.  Studied
  subsequently as a property in its own right, Property $A$ for a
  discrete group is known to be equivalent to exactness of the
  reduced group $C^*$-algebra and to the amenability of the action of the
  group on its Stone-Cech compactification.  In this paper we study
  exactness for groups acting on a finite dimensional $\CAT(0)$ cube
  complex.  We apply our methods to show that Artin groups of type FC
  are exact.  While many discrete groups are known to be exact the
  question of whether every Artin group is exact remains open.
\end{abstract}

\maketitle

\section{Introduction}

A discrete metric space $X$ has {\it Property $A$\/} if there exists a
sequence of families of finitely supported probability measures
$f_{n,x}\in\ell^1(X)$, indexed by $x\in X$, and a sequence of
constants $S_n>0$, such that:
\begin{nlist}
  \item  For every $n$ and $x$ the function $f_{n,x}$ is supported in
                 $B_{S_n}(x)$.
  \item  For every $R>0$, we have
\begin{equation*}
     \normi{f_{n,x}-f_{n,x'}}\to 0
\end{equation*}
uniformly on the set $\{(x,x') : d(x,x')\leq R\}$ as $n\to \infty$.
\end{nlist}
A discrete group has Property $A$ if its underlying proper
metric space does (this is independent of the choice of proper
metric).  In this case the definition is
recognized as a non-equivariant form of the Reiter condition
for amenability.  

For groups it transpires that Property $A$ is equivalent to a wide
variety of other conditions including exactness of the reduced group
$C^*$-algebra, $C^*$-exactness of the group itself (defined in terms
of crossed products) and amenability of the action of the group on its
Stone-Cech compactification \cite{higson-roe,ozawa}.  The class of groups
possessing Property $A$ is large and diverse -- for example, it
contains every amenable group, every linear group and every hyperbolic
group, and is closed under many natural operations \cite{GHW, kirchberg-wassermann,DG}.  In
this article we shall for groups use the terms {\it Property $A$\/}
and {\it exactness\/} interchangeably.

In previous work, in collaboration with J.~Brodzki, S.~Campbell and
N.~Wright, we showed that a finite dimensional $\CAT(0)$ cube complex
has Property $A$ \cite{BCGNW}.  For the proof we constructed an
explicit family of weight functions which, when suitably normalised,
become the functions $f_{n,x}$ in the definition above.  As a
consequence a group acting (metrically) properly on a finite
dimensional $\CAT(0)$ cube complex is exact. In particular all
finitely generated right-angled Artin groups are exact. Since an infinitely generated Artin group is the ascending union of its finitely generated parabolic subgroups any countable 
right angled Artin group is exact.

In this paper we carry out an extended study of the weight functions,
as defined on a suitable compact space combinatorially defined in
terms of the hyperplanes and half spaces of the complex.  Our analysis
of their topological and measure theoretic properties leads to a new
inheritance property for exact groups.  Indeed, while true that a
group acting on a {\it locally finite\/} Property $A$ space with exact
stabilisers is exact, the analogous statement for general non-locally
finite spaces is false.  In order to guarantee inheritence in the more
general context one needs to assert control over coarse stabilisers --
point stabilisers are not sufficient. In our setting, following an
idea of Ozawa \cite{ozawa}, extension of the weight functions to a compact
space affords the required additional control.  We obtain the
following result:

\begin{thma}
  A countable discrete group acting on a finite
  dimensional $\CAT(0)$ cube complex is exact if and
  only if each vertex stabilizer of the action is exact.
\end{thma}

As an application, we offer the following result in which we do not assume the Artin group is finitely generated.

\begin{thmb}
An Artin group of type FC is exact.
\end{thmb}

We note that Altobelli characterised the Artin groups of type FC as
the smallest class of Artin groups containing the Artin groups of
finite type which is closed under amalgamations along parabolic
subgroups \cite{Altobelli}.  Thus, this theorem could alternately be
obtained by appealing to the stability theorem for graph products of
exact groups first established in \cite{G}.  (See also \cite{DG} for a
more modern discussion.) However the class of groups acting on
$\CAT(0)$ cube complexes is considerably richer than the class of
groups acting on trees and we expect Theorem A to have many other
applications.

The paper is organised as follows. In Section~\ref{cubical} we recall
the definition and basic properties of a $\CAT(0)$ cube complex, with an
emphasis on the combinatorics of vertices, hyperplanes and half
spaces. We describe a compact space in which the vertex set of the
complex embeds, and give an explicit description of the points of this
space. In Section~\ref{weightfunctions} we recall the definition of
the weight functions from \cite{BCGNW} and analyse their topological
and measure theoretic properties.  In Section~\ref{sec:permanence},
adapting slightly the method of Ozawa \cite{ozawa}, we establish
Theorem~A.  Section~\ref{artin} contains relevant background on Artin
groups, and a discussion of Theorem~B.

\section{Cubical complexes}\label{cubical}

A {\it $\CAT(0)$ cube complex\/} is a cell complex in which each cell
is a Euclidean cube of side length $1$ and the attaching maps are
isometries; the complex is equipped in the usual way with a
geodesic metric which is required to satisfy the $\CAT(0)$ condition
of non-positive curvature. It follows that a $\CAT(0)$ cube complex is
simply connected, even contractible, as a topological space. 

The midpoint of each edge of a $\CAT(0)$ cube complex defines a {\it
  hyperplane\/} -- the union of all geodesics passing through the
midpoint at right angles to the underlying edge, the angle being
measured in the local Euclidean metric on each cube. Each hyperplane
is a totally geodesic codimension one subspace which is locally
separating, and therefore globally separating since the complex is
simply connected.

In this paper we shall be concerned exclusively with the combinatorics
of the vertices, hyperplanes and half spaces of a $\CAT(0)$ cube
complex.  We shall now outline the facts we require -- we refer to
\cite{ChatterjiNiblo}, and the standard references \cite{BH,roller}
for additional details. 

Let $X$ be a $\CAT(0)$ cube complex.  Slightly abusing notation, we
shall denote the set of vertices of the complex by $X$ as well. Each
hyperplane decomposes the vertex set into two subsets, the two {\it
  half spaces} determined by the hyperplane. There is {\it a priori\/} no
reason to prefer one of these half spaces over the other and we shall
adopt the following convention: {\it fix\/} a base vertex, and for
each hyperplane $H$ denote by $H^+$ the half space containing the base
vertex; the complementary half space is denoted $H^-$. A hyperplane
$H$ {\it separates\/} two vertices if one belongs to $H^+$ and the
other to $H^-$.

Let $x$ and $y$ be vertices in $X$.  The {\it interval between $x$ and
  $y$\/} is the intersection of the half spaces containing both $x$
and $y$; it is a finite set which we shall denote $[x,y]$.  It follows
directly from the definition that a vertex belongs to $[x,y]$
exactly when there are no hyperplanes which separate it from both $x$
and $y$.  A useful alternate description of the interval is the
following: $[x,y]$ consists of those vertices in $X$ which lie on an
edge geodesic joining $x$ and $y$.  Observe that $[x,x] = \{\, x \,\}$.
Given three vertices $x,y,z$ the intersection 
$[x,y]\cap [y,z]\cap [z,x]$ is comprised of a single vertex; we denote
this vertex by $m(x,y,z)$ and refer to it as the
{\it median\/} of $x$, $y$, $z$.

Let now $\h$ denote the set of hyperplanes in $X$.  Each vertex $x$
determines a function $\h\to \{\, \pm 1 \,\}$ according to the rule
\begin{equation*}
  x(H) = \begin{cases}
             +1, & x\in H^+ \\
             -1, & x\in H^-.
         \end{cases}
\end{equation*}
Observe that $x(H)= -1$ precisely when $H$ separates $x$ and the
fixed base vertex. 
While the notation appears clumsy, it is chosen for convenience in the
following statement: for every vertex $x$ and hyperplane $H$ we see
that $x$ belongs to the half space $H^{x(H)}$. (Here, we are
implicitly writing $H^{+1}$ for $H^+$, and similarly for $H^-$.)
We denote by $\{\, \pm 1 \,\}^{\h}$ the {\it Hamming cube\/} on 
$\mathcal H$,
that is, the set of 
functions $\h\to \{\, \pm 1 \,\}$ equipped with the infinite product
topology. We obtain by the above  a map
\begin{equation*}
   X \to  \{\, \pm 1 \,\}^{\h} \,.
\end{equation*}
Any two (distinct) vertices are separated by at least one
hyperplane and if $H$ separates $x$ and $y$ then $x(H)\neq y(H)$.
Thus, this map is injective.  We identify $X$ with its image, the
subset of {\it original vertices\/}.

An element $z$ of the Hamming cube 
is an {\it admissible vertex\/}
if for every two hyperplanes $H$ and $K$ there exists an original
vertex $x$ for which both $x(H)=z(H)$ and $x(K)=z(K)$.  Equivalently,
$z$ is admissible if for every $H$ and $K$ the half spaces $H^{z(H)}$
and $K^{z(K)}$ have non-empty intersection.  Clearly, an original
vertex is admissible.  Admissible vertices that are {\it not\/}
original vertices are  {\it  ideal vertices\/}.

We pause briefly to consider an example.

\begin{ex} 
  The Euclidean plane equipped with its usual integer lattice squaring
  is a $\CAT(0)$ cube complex of dimension two.  The vertices are the
  integer lattice points. The hyperplanes are the horizontal and
  vertical lines intersecting the axes at half-integer points:
  \begin{align*}
    H_n \quad &\colon \quad y=n+\tfrac{1}{2} \\
    K_n \quad &\colon \quad x=n+\tfrac{1}{2},
 \end{align*}
 for an integer $n$.  Fix $(0,0)$ as the base vertex.  

The lattice point $(p,q)$ in the first quadrant defines an original
vertex by 
\begin{equation*}
  (p,q)(H_n) = \begin{cases}
   -1, & \text{when $0\leq n< q$} \\
   +1, &\text{else},
 \end{cases}
\qquad
  (p,q)(K_n) = \begin{cases}
   -1, & \text{when $0\leq n< p$} \\
   +1, &\text{else}.
 \end{cases}
\end{equation*}
There are also ideal vertices.  For example, we may orient all
horizontal lines to point upwards and all vertical lines to point to
the right defining an admissable vertex $(+\infty, +\infty)$
by 
\begin{equation*}
  (+\infty, +\infty)(H_n) =
  \begin{cases}
    +1, & \text{when $n< 0$} \\ -1, & \text{else},
  \end{cases}
\qquad 
  (+\infty, +\infty)(K_n) =
  \begin{cases}
    +1, & \text{when $n< 0$} \\ -1, & \text{else}.
  \end{cases}  
\end{equation*}
We shall think of this vertex as ``the top right corner'' of the
plane.  The full set of ideal vertices comprises four corner points,
and four lines -- one each at the East, West, North and South of the
plane -- as illustrated in Figure~\ref{fig:plane} below.
\end{ex}

\begin{figure}[h] 
    \centering
    \includegraphics[width=5in]{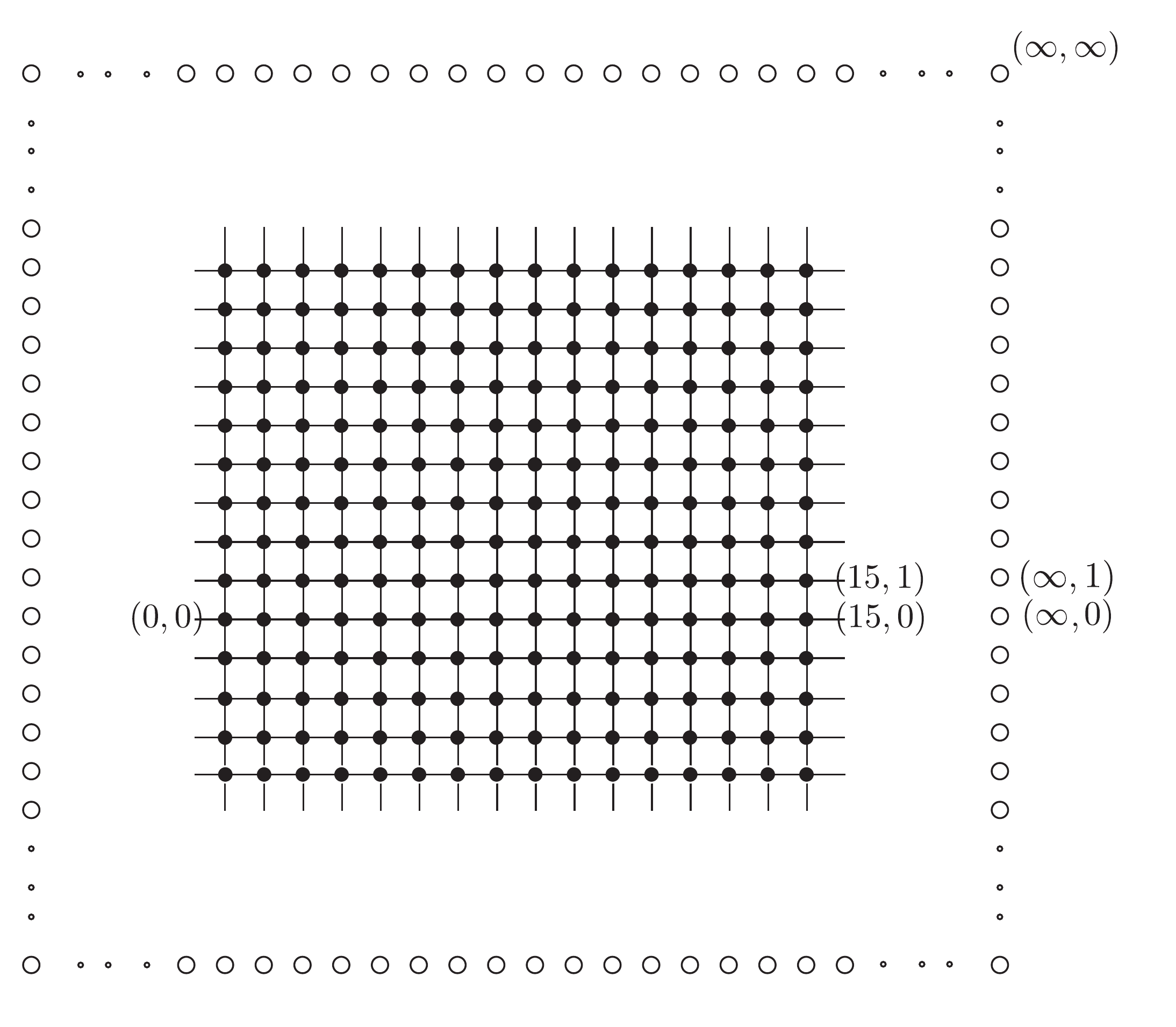} 
    \caption{The Euclidean plane with admissable vertices attached}
    \label{fig:plane}
\end{figure}

\begin{lem}\label{intersects}
  An element $z$ of the Hamming cube is an admissible vertex if and
  only if for every  $n\geq 2$ and every collection of $n$
  hyperplanes $H_1,\dots,H_n$ there exists an original vertex $x$
  satisfying $x(H_i)=z(H_i)$, for each $i=1,\dots,n$. 
\end{lem}
\begin{proof}
  We are concerned with the forward implication, which we prove by
  induction on $n$.  The case $n=2$ is covered by the definition of an
  admissible vertex.   Let $n>2$ and let $H_1,\dots,H_n$
  be $n$ hyperplanes.  By the induction hypothesis we have an original
  vertex $x_1$ agreeing with $z$ on
  $H_1,\dots,H_{n-1}$, another original vertex $x_2$ agreeing with $z$
  on $H_2,\dots,H_n$, 
  and a third original vertex $x_3$ agreeing with $z$ on $H_1$ and
  $H_n$.  The median $m(x_1, x_2, x_3)$ has the desired property.  
\end{proof}

\begin{prop}\label{closure}
  The closure $\overline{X}$ of the set of original vertices is the
  set of admissible vertices. 
\end{prop}

\begin{proof}
  Lemma~\ref{intersects} shows that every basic open neighborhood of an
  admissible vertex contains an original vertex. Thus, every
  admissible vertex is in the closure of the original vertices.

  Conversely, suppose $z$ belongs to the closure of the original
  vertices, and let $H$ and $K$ be hyperplanes. The requirements
  $x(H)=z(H)$ and $x(K)=z(K)$ define an open neighborhood of $z$ in
  the infinite product, so must contain an original vertex. Hence $z$
  is admissable.
\end{proof}

We extend the above terminology regarding hyperplanes
and half spaces to $\overline X$ in the obvious way.  For example, an
admissible vertex $z$ belongs to the half space $H^+$ if $z(H)=+1$; it
belongs to $H^-$ if $z(H)=-1$.  Thus, we extend the half spaces to
include ideal vertices.
Having extended the notion of half space
to the set of admissable vertices we define intervals exactly as
before, as intersections of half spaces.  In later sections we shall
work onlywith intervals $[x,z]$ in which $x$ is an original vertex,
whereas $z$ may be either an original or an ideal vertex.

A pair of admissible vertices $z$ and $w$ are separated by the
hyperplane $H$ when $z(H)\neq w(H)$.  While only finitely many
hyperplanes may separate a pair of original vertices, a pair of
vertices at least one of which is ideal may be separated by infinitely
many hyperplanes. For example, in Figure~\ref{fig:plane} the
ideal vertices $(\infty, 0)$ and $(\infty, 1)$ are separated by a
single hyperplane, whereas the ideal vertices $(\infty, 0)$ and
$(\infty, \infty)$ are separated by infinitely many horizontal
hyperplanes.

A pair of admissable vertices $z$ and $w$ are {\it adjacent across the
  hyperplane $H$\/} if they differ only on $H$.  An admissable vertex
$z$ is {\it adjacent to the hyperplane $H$} if there is an admissable
vertex $w$ such that $z$ is adjacent to $w$ across $H$.

\begin{prop}
  Let $x$, $y$ and $z$ be admissible vertices.  The element of the
  Hamming cube defined by
  \begin{equation*}
    m(H) = \begin{cases}
      +1, &\text{at least two of\/ $x(H)$, $y(H)$, $z(H)=1$} \\
      -1, &\text{at least two of\/ $x(H)$, $y(H)$, $z(H)=-1$}
    \end{cases}
  \end{equation*}
is an admissible vertex.  It is the unique admissible vertex
belonging to all three of the intervals $[x,y]$, $[y,z]$ and $[x,z]$.
\end{prop}

\begin{proof}
  We first check that $m$ is admissible.  Suppose hyperplanes $H$ and
  $K$ are given.  At least two of the vertices $x$, $y$, $z$ must
  agree with $m$ on $H$ and at least two must agree with $m$ on $K$,
  so at least one agrees with $m$ on both $H$ and $K$. Since that
  vertex is itself admissable there is an original vertex which agrees
  with $m$ on both $H$ and $K$.

We next check that $m$ belongs to the interval $[x,y]$.  Indeed, if
$H$ separates $m$ from both $x$ and $y$ then $x(H) = y(H)\neq m(H)$,
contradicting the definition of $m$.  The other intervals are treated
similarly.

Finally, we verify uniqueness.  Suppose $m'$ is an admissible vertex
belonging to each of the intervals $[x,y]$, $[y,z]$ and $[x,z]$.
Given a hyperplane $H$ at least two of the vertices
$x$, $y$ and $z$ belong to a common half space of $H$.  Thus, $m'$
agrees with at least two of the vertices $x$, $y$ and $z$ on $H$ so
that $m'$ agrees with $m$ on $H$ as well.  As the hyperplane $H$ was
arbitrary, we conclude that $m'=m$.
\end{proof}

The proposition extends the notion of median to admissible vertices: 
the admissible vertex $m$ described in the statement is the {\it
  median\/} of the three admissable vertices $x$, $y$ and $z$; as with
medians of original vertices we write $m=m(x,y,z)$.

We close this section with some elementary remarks concerning the
topological space $\overline X$. Each half space is a clopen set. The
collection of finite intersections of half spaces comprises a basis
for the topology on $\overline X$.  For an admissible vertex $z$, the
singleton $\{\, z \,\}$ is closed; if $z$ is an ideal vertex 
$\{\, z \,\}$ is {\it not\/} open.  For original vertices the
situation is more complicated.

\begin{prop}
  Let $x$ be an original vertex.  The following are equivalent:
  \begin{nlist}
    \item $\{\, x \,\}$ is open in $\overline X$;
    \item $\{\, x \,\}$ is open in $X$ with respect to  the subspace topology;
    \item $x$ is a finite vertex.
  \end{nlist}
\end{prop}

\noindent
Here, an {\it original\/} vertex is  said to be {\it finite\/} if there
are only finitely many hyperplanes adjacent to it.
It follows that, in the case of a non-locally finite complex, $X$ itself
has non-trivial topology as a subspace of $\overline X$ -- that is, the
subspace topology on $X$ is not discrete.

\begin{proof}
  Elementary topology shows that (1) implies (2), and (2) implies (3).
  If $x$ is a finite vertex, and $H_1,\dots,H_n$ are the (finitely
  many) hyperplanes adjacent to $x$ then we claim that 
  \begin{equation}
  \label{singleton}
    \{\, x \,\} = H_1^{x(H_1)} \cap \dots \cap H_n^{x(H_n)},
  \end{equation}
  which is a basic open set for the topology on $\overline X$.  To
  verify (\ref{singleton}) we must, according to our conventions, show
  that no admissible vertex other than $x$ can belong to the displayed
  intersection of half spaces.  
  It is an elementary fact that the intersection
  can contain no original vertex other than $x$.  Thus, we must show
  that the intersection can contain no ideal vertex.
  Suppose that $z$ is an ideal vertex which agrees with $x$ on the
  given hyperplanes.  Necessarily, $z$ differs from $x$ on some other
  hyperplane $K$.  By Lemma~\ref{intersects} there is an original
  vertex $y$ which agrees with $z$ on the hyperplanes $H_1,\dots, H_n$,
  and also on $K$.  Thus, $y$ is an original vertex that agrees with
  $x$ on $H_1,\dots,H_n$ but differs from it on $K$, a contradiction.
\end{proof}

While $\overline X$ is a compact space containing $X$ as a dense
subspace, it is {\it not\/} in general a compactification of $X$ in the classical
sense -- when $X$ is not locally finite it {\it need not\/} be an open
subset of $\overline X$.  We shall not require this fact below, and
its verification is left to the reader.  (But, compare to the
discussion surrounding
Propositions~\ref{middle} and \ref{notcont}.)

\begin{prop} 
  The compact space $\overline X$ contains $X$ as a dense subspace.  
  An action of a discrete group on $X$ by cellular automorphisms
  extends to an action
  on $\overline X$ by homeomorphisms. 
\end{prop}
\begin{proof}
  Open sets in $\overline X$ are unions of finite intersections
  of half spaces all of which contain original vertices by Lemma
  \ref{intersects}, so $X$ is  dense in $\overline{X}$ as required.
  An automorphism of $X$ preserves the half space structure and
  therefore extends to a homeomorphism of
  $\overline X$. 
\end{proof}

\section{Weight functions}\label{weightfunctions}

Let $X$ be the vertex set of a {\it finite dimensional\/} $\CAT(0)$
cube complex. In previous work we constructed weight functions on $X$
-- we used these to show that $X$ has Property $A$, when viewed as a
metric space with either of its natural metrics \cite{BCGNW}. We shall use the
previously constructed weight functions in the present context as
well, and now recall their definition.

Fix an ambient dimension $N$ greater than or equal to the dimension of
the complex. For every $z\in \overline X$ and every vertex $a\in
[x,z]\cap X$ the {\it deficiency of\/ $a$ \textup{(}relative to the
  interval $[x,z]$\textup{)}\/} is 
\begin{equation}
\label{def}
  \delta_{[x,z]}(a)=N-k, 
\end{equation}
where $k$ is the number of hyperplanes cutting edges adjacent to $a$
and which separate $a$ (and hence also $x$) from $z$. By hypothesis
$0\leq\delta_{[x,z]}(a)\leq N$. Now for every vertex $x\in X$ and
every $z\in \overline X$ we define the {\it weight function\/}
$\phi^n_{x,z}$ according to the formula 
\begin{equation}
\label{weights}
  \phi^n_{x,z}(a) = \begin{cases}
           {n-d(x,a)+\delta_{[x,z]}(a) \choose\delta_{[x,z]}(a)}, &a\in [x,z] \\
           0, &a\notin [x,z].
  \end{cases}
\end{equation}
Intuitively $\phi^n_{x,z}$ measures the flow of a mass placed at the
vertex $x$ as it flows towards $z$ with $n$ playing the role of the
time parameter.  The basic properties of the weight functions are
summarized in the following theorem \cite{BCGNW}.  In the statement,
$B_n(x)$ denotes the {\it ball of radius $n$ and center $x$\/},
comprised of those (original) vertices separated from $x$ by at most
$n$ hyperplanes; the norms are $\ell_1$-norms.

\begin{thm}
\label{oldthm}
  Let $X$ be the vertex set of a finite dimensional $\CAT(0)$ cube
  complex, and let $\overline X$ be the compact space of admissible
  vertices, defined previously. Fix an ambient dimension $N$ not less
  than the dimension of $X$.  The weight functions
  \begin{equation*}
    \phi^n: X \times \overline X \to \ell_1(X), 
              \quad (x,z)\mapsto \phi^n_{x,z}
  \end{equation*}
defined by formula \textup{(}\ref{weights}\textup{)} satisfy the following: 
\begin{nlist}
  \item $\phi^n_{x,z}$ is $\N\cup \{0\}$-valued;
  \item $\phi^n_{x,z}$ is supported in $B_n(x)\cap [x,z]$;
  \item $\| \phi^n_{x,z} \| = {n+N\choose N}$;
  \item if $x$ and $x'\in X$ are adjacent then 
            $\| \phi^n_{x,z}-\phi^n_{x'z} \| = 2 {n+N-1\choose N-1}$.
\end{nlist}
Further, if a discrete group $\Gamma$ acts cellularly on $X$, hence
also by homeomorphisms on $\overline X$, we have 
\begin{nlist}\setcounter{nctr}{4}
  \item $s\cdot \phi^n_{x,z} = \phi^n_{sx,sz}$ ,
\end{nlist}
for every $s\in\Gamma$.
\end{thm}

\begin{proof}
  Properties, (1) and (2) are immediate from the defining formula
  (\ref{weights}). 
  Property (5) is also apparent from the defining formula -- indeed,
  it is equivalent to the assertion that
  \begin{equation*}
    \phi^n_{x,z}(a) = \phi^n_{sx,sz}(sa),
  \end{equation*}
  for all $x$, $a\in X$, $z\in \overline X$ and $s\in \Gamma$, which holds
  since $\Gamma$ acts cellularly and the weight functions are
  determined by the combinatorics of hyperplanes.
  Finally, properties (3) and (4) are established in Propositions 2.3
  and 2.4 of \cite{BCGNW}. 
\end{proof}

The remainder of the section is devoted to an analysis of the
continuity properties of the weight functions defined in
(\ref{weights}).  In particular, we shall view $\phi^n_{x,z}(a)$, as a
function of  
$z\in\overline X$,  {\it for a fixed natural number\/ $n$, and for
  fixed $x$ and $a\in X$\/}.  Our first result in this direction is
the following proposition.

\begin{prop}
\label{borel}
  Fix a natural number $n$, and original vertices $x$ and $a\in X$.
  The function
  \begin{equation*}
    \Phi : \overline X \to \N, \qquad
      \Phi(z) = \phi^n_{x,z}(a).
  \end{equation*}
satisfies the following:
\begin{nlist}
  \item if $n\leq d(x,a)$ then $\Phi$ is continuous;
  \item if $n> d(x,a)$ and
    \begin{ilist}
          \item $a$ is finite then $\Phi$ is continuous;
          \item $a$ is {\it not\/} finite then $\Phi$ is Borel.
    \end{ilist}
\end{nlist}
\end{prop}

Before turning to the proof of the proposition, we require a lemma.

\begin{lem}
\label{zeropreimage}
 For any choice of original vertices $x$ and $a$ the set 
 $\{\, z : a\in [x,z] \,\}$ is clopen in $\overline X$.
\end{lem}
\begin{proof}
 The complement of the set in question is
 \begin{equation*}
   \{\, z : \text{$\exists$ $H\in\h$ such that $H$ 
              separates $a$ from both $x$ and $z$} \,\} =
        \cup H^{x(H)},
 \end{equation*}
where the union is over the finite set of hyperplanes
separating $a$ from $x$.  (When $a=x$ this set is empty.)
This set is clopen, hence so is its complement.
\end{proof}

\begin{proof}[Proof of Proposition~\ref{borel}]
  We divide (1) into two cases.  First, if $n<d(x,a)$ then $\Phi$ is
  identically zero.  Second, if $n=d(x,a)$ then
  $\Phi$ is given by the formula
  \begin{equation*}
    \Phi(z) = \begin{cases}
          1, & a\in [x,z] \\ 0, &\text{else}.
              \end{cases}
  \end{equation*}
In other words, $\Phi$ is the characteristic function of the clopen
set appearing in the previous lemma, and so it is continuous.

We consider (2a) and (2b) simultaneously, and proceed by analysing the
level sets of $\Phi$.

Write $A=n-d(x,a)>0$.
Inspecting (\ref{weights}) we see that $\Phi$ is given by the formula
\begin{equation*}
  \Phi(z) = \begin{cases}
         {A+(N-k)\choose (N-k)}, &a\in [x,z] \\
           0, &a\notin [x,z], \end{cases}
\end{equation*}
where $k=k(z)$ appears in the formula (\ref{def}) for the deficiency.
Thus, the values of $\Phi$ are among the (distinct) natural numbers
\begin{equation*}
  0 \quad\text{and} \quad {A+(N-k)\choose (N-k)}, 
         \quad 0\leq k \leq \dim(X)\leq N.
\end{equation*}
Further, the level sets corresponding to these values are
$\Phi^{-1}(0) = \{\, z : a\notin [x,z] \,\}$ and
\begin{equation}
\label{levelset}
 \Phi^{-1}\left({A+(N-k)\choose (N-k)}\right) =   
      \{\, z : \text{$a\in [x,z]$ and $\delta_{[x,z]}(a)=N-k$} \,\},
\end{equation}
respectively.  The first of these is clopen, by the lemma.  We analyze
the second (\ref{levelset}).  

Let $K_1,\dots,K_n$ be the (finitely many) hyperplanes separating $x$
and $a$.  Let $H_1,H_2,\dots$ be the hyperplanes adjacent to $a$ and
{\it not\/} separating $x$ and $a$.  Observe the the collection of
$H$'s is finite exactly when $a$ is a finite vertex.  The conditions
defining the level set (\ref{levelset}) are that $x$ and $z$ are
separated by every $K_i$ and {\it exactly\/} $k$ of the $H_j$.
Similarly,
\begin{equation}
\label{superlevel}
  \Phi(z) > {A+(N-k)\choose (N-k)}
\end{equation}
precisely when $x$ and $z$ are separated by every $K_i$ and 
{\it fewer than\/} $k$ of the $H_j$.  Thus, the set of admissible $z$
satisfying (\ref{superlevel}) is precisely
\begin{equation}
\label{superlevelset}
  K_1^{a(K_1)}\cap\dots\cap K_n^{a(K_n)} \cap
      \bigcap \left( H_{j_1}^{x(H_{j_1})} \cup\dots\cup  H_{j_k}^{x(H_{j_k})}\right), 
\end{equation}
with the large intersection being over the $k$ element subsets
$j_1,\dots,j_k$ of $j$'s.  The set appearing in (\ref{superlevelset})
is closed so that, as the difference of two closed sets, the level set
(\ref{levelset}) is Borel, as is $\Phi$.
Further, if $a$ is finite, the set appearing in (\ref{superlevelset})
is clopen -- the intersection is finite because there are only
finitely many $k$ element subsets of $j$'s.  In this case, as the
difference of clopen sets, the level
set (\ref{levelset}) is clopen and
$\Phi$ is continuous.
\end{proof}

\begin{rem}
  In the course of the proof we have established the following fact:
  for all choices of the parameters $n$, $x$ and $a$, if $\Phi(z)=0$
  then $\Phi$ is continuous at $z$.
\end{rem}

\begin{rem}
The proposition leaves open the question of whether $\Phi$ is
continuous when $a$ is an infinite point.  Indeed, it is not difficult
to see that if $a$ is infinite then $\Phi$ is {\it not\/} continuous.

  In the notation of the proposition, suppose that $a$ is an infinite
  point (and also that $n> d(x,a)$).  We show that $\Phi$ is {\it not\/} continuous
  at $a$.  Indeed, let $H_1,H_2,\dots$ be an infinite sequence of
  hyperplanes adjacent to $a$, none of which separate $a$ from $x$.
  Let $z_j$ be the vertex immediately across $H_j$ from $a$, and note
  that $a\in[x,z_j]$.
  Inspecting the definition (\ref{weights}) we see that 
  \begin{equation*}
    \Phi(a) = {n-d(x,a)+(N-0)\choose (N-0)} \neq 
       {n-d(x,a) +(N-1)\choose N-1} = \Phi(z_j).
  \end{equation*}
The value $\Phi(z_j)$ is independent of $j$, different from $\Phi(a)$
and $z_j\to a$.
\end{rem}

While this remark is quite simple, it leads to a complete analysis of
the continuity of the $\Phi$, which we develop in the two subsequent
propositions.  Note that when $z=a$ the first of these is
essentially the previous remark.

\begin{prop}
\label{middle}
  Continue in the notation of Proposition~\ref{borel}, and assume
  $n>d(x,a)$.  Let $z$ be an original vertex for which 
  $\Phi(z)\neq 0$.  The function $\Phi$ is continuous at $z$ exactly
  when only finitely many hyperplanes are adjacent to both $a$ and
  $z$. 
\end{prop}
\begin{proof}
  The forward implication proceeds exactly as the remark.  Indeed,
  with $z$ as in the statement, let $H_1,H_2,\dots$ be an infinite
  sequence of hyperplanes adjacent to both $a$ and $z$, none of which
  separate $a$ from $z$, and none of which separate $a$ from $x$.  The
  vertices $z_j$ immediately across $H_j$ from $z$ witness the
  non-continuity of $\Phi$ at $z$.

For the reverse implication, let $H_1,\dots,H_n$ be the hyperplanes
adjacent to both $a$ and $z$, and let $K_1,\dots,K_m$ be the
hyperplanes that separate $x$ and $z$.  The intersection
\begin{equation*}
  H_1^{z(H_1)} \cap\dots\cap H_n^{z(H_n)} \cap
    K_1^{z(K_1)} \cap\dots\cap K_m^{z(K_m)}
\end{equation*}
is a clopen neighborhood of $z$.  Let $w$ belong to this neighborhood.
We claim that $\Phi(w)=\Phi(z)$.  Now, since $w\in K_i^{z(K_i)}$ for
all $i$ we have $a\in [x,w]$.  Thus, the values $\Phi(w)$ and
$\Phi(z)$ are given by the first case in (\ref{weights}) and we must
show
\begin{equation*}
  \delta_{[x,z]}(a) = \delta_{[x,w]}(a).
\end{equation*}
We introduce the notation $\mathfrak N_z(a)$ for the {\it deficiency
  set of $a$ with respect to $z$}, that is, the set of hyperplanes
that are adjacent to $a$ and that separate $a$ from $z$.  The
deficiency $\delta_{[x,z]}(a)$ is the difference of $N$ and the
cardinality of $\mathfrak N_z(a)$.  Thus, it suffices to
show that $\mathfrak N_z(a) = \mathfrak N_w(a)$.

Because $a\in [x,z]$ (by hypothesis), a hyperplane separating $a$ and
$z$ is one of the $K_i$, which therefore also separates $a$ from $w$.
It follows that $\mathfrak N_z(a)\subset \mathfrak N_w(a)$.  For the
reverse inclusion, suppose $H\in \mathfrak N_w(a)$.  We must show that
$H$ separates $a$ from $z$.  If not, then the subsequent lemma shows
that $H$ is adjacent to $z$ -- indeed, $z\in[a,w]$ since any hyperplane
separating $a$ from $z$ also separates $x$ from $z$, thus is among the
$K_i$.  Thus, $H$ is one of the $H_j$, so that $z$ and $w$ are on the
same side of $H$, the side opposite $a$.  This is a contradiction.
\end{proof}

\begin{lem}
\label{slice}
  Suppose that $H$ is adjacent to $a$, that $H$ separates $z$ and $w$,
  and that $z\in[a,w]$.  Then $z$ is adjacent to $H$.
\end{lem}
\begin{proof}
  Observe that $H$ separates $a$ from $w$, and hence not from $z$;
  indeed, otherwise $H$ separates both $a$ and $w$ from $z$
  contradicing $z\in [a,w]$.  
  Let $b$ be the vertex immediately across $H$ from $a$.  Let $m$ be
  the median of $b$, $z$ and $w$.  We claim that $H$ is the unique
  hyperplane separating $z$ from $m$ so that, in particular, $z$ is
  adjacent to $H$.  Indeed,
  \begin{equation*}
    w(H) = b(H) \neq a(H) = z(H)
  \end{equation*}
shows that $m(H) \neq z(H)$, that is, $H$ separates $m$ and $z$.
Further, if $K$ is such that $z(K)\neq m(K)$ then 
\begin{equation*}
  b(K)=w(K)=m(K) \neq z(K) = a(K),
\end{equation*}
where the last equality holds since $a\in[a,w]$.  Thus, $K$ separates
$a$ from $b$, and $K=H$.
\end{proof}

Continuity of $\Phi$ at ideal vertices is slightly more subtle, and is
treated in the next proposition.  Observe that when $z$ is an original
vertex, the stated condition reduces to the one in the previous
proposition -- indeed, when $z$ is an original vertex elements of the
interval $[a,z]$ can only be separated from $z$ by those (finitely
many) hyperplanes that separate $a$ from $z$; thus, any sequence of
such points converging to $z$ is eventually constant.

\begin{prop}
\label{notcont}
  Continue in the notation of Proposition~\ref{borel}, and assume
  $n>d(x,a)$.  Let $z$ be an admissible vertex for which $\Phi(z)\neq
  0$.  The function $\Phi$ is not continuous at $z$ precisely when
  there is a sequence $m_1,m_2,\dots$ of
  admissable vertices in the interval 
  $[a,z]$ converging to $z$ and a sequence $H_1,H_2,\dots$ of distinct
  hyperplanes adjacent to $a$ for which
  $H_j$ is adjacent to $m_j$.
\end{prop}

\begin{proof}
  We provide Figure~\ref{fig:example} to aid the reader in following
  the proof.  
  
  Suppose first that $\Phi$ is not continuous at $z$, and
  that $\Phi(z)\neq 0$.  We claim that there exists a sequence of
  admissible vertices $z_j\to z$ such that every $z_j$ satisfies the
  following:
  \begin{nlist}
    \item $a\in [x,z_j]$
    \item $\delta_{[x,z_j]}(a) \neq \delta_{[x,z]}(a)$.
  \end{nlist}
  Indeed, begin with a sequence $z_j\to z$ for which $\Phi(z_j)$ does
  not converge to $\Phi(z)$.  Now, every sequence of admissible
  vertices converging to $z$ must satisfy (1) on a tail --
  $\Phi(z)\neq 0$ implies that $z$ belongs to the clopen set described
  in Lemma~\ref{zeropreimage}.  Thus, we may assume our sequence
  satisfies (1), so that the values $\Phi(z_j)$ are given by the first
  case in (\ref{weights}).  Thus, $\delta_{[x,z_j]}(a)$ does not
  converge to $\delta_{[x,z]}(a)$ and, we arrange for (2) by passing
  to a subsequence.\footnote{As the deficiency can assume only
    finitely many values, we could also arrange that the
    $\delta_{[x,z_j]}(a)$ is constant (independent of $j$) and
    different from $\delta_{[x,z]}(a)$.}

  Consider now the median 
  \begin{equation}
  \label{mj}
    m_j=m(a,z_j,z),
  \end{equation}
  which by construction lies in the interval $[a,z]$.  As shown in
  \cite{BCGNW} the sequence $m_j$ converges to $z$. The $m_j$ (rather,
  a subsequence) will be the sequence we seek -- it remains to locate
  the required adjacent hyperplanes.  To do this, we claim that for
  sufficiently large $j$, we have 
  $\mathfrak N_{m_j}(a)\neq\mathfrak N_{z_j}(a)$ -- here, we employ
  the notation regarding deficiency sets introduced in the proof of
  Proposition~\ref{middle}. Again as shown in \cite{BCGNW}, since the
  $m_j$ converge to $z$ and lie in the interval $[a,z]$, the
  subsets $\mathfrak N_{m_j}(a)$ eventually stabilise at 
  $\mathfrak N_z(a)$.  Thus, combined with (2) we see that
  for sufficiently large $j$
  \begin{equation*}
         \delta_{[x,m_j]}(a)=\delta_{[x,z]}(a)\neq \delta_{[x,z_j]}(a),
  \end{equation*}
from which the claim follows.
Thus, for each sufficiently large $j$ there is a
hyperplane $H_j$ adjacent to $a$ that separates $m_j$ and $z_j$. It
follows from Lemma~\ref{slice} that $H_j$ is adjacent to $m_j$ -- by
(\ref{mj}) we have $m_j\in [a,z_j]$ so that the lemma applies.

It remains only to see that the sequence $H_j$ contains infinitely
many distinct hyperplanes.  Indeed, we shall show slightly more --
that every hyperplane $H$ can appear as an $H_j$ only finitely many
times.  Assume to the contrary, that the hyperplane $H$ appears
infinitely many times.  Then, since both $z_j$ and $m_j$ converge
to $z$, they are eventually on a common side of $H$, which contradicts
the fact that $H$ separates $m_j$ and $z_j$.

Suppose now that $\Phi(z)\neq 0$ and that the conditions in the
statement are satisfied.  We are to show that $\Phi$ is not continuous
at $z$.  As remarked above, since the $m_j$ converge to $z$ and all
belong to the interval $[a,z]$, the deficiency sets $\mathfrak
N_{m_j}(a)$ eventually stabilise at $\mathfrak N_z(a)$ \cite{BCGNW};
without loss of generality we may assume that they all coincide.  Let
$m_j'$ denote the vertex immediately across $H_j$ from $m_j$.  We
claim that $m'_j$ converges to $z$.  To see this, let $K$ be an
arbitrary hyperplane.  If $K$ is not one of the $H_j$ then
$m'_j$ and $m_j$ agree on $K$ for every $j$; if $K$ is
one of the $H_j$ then $m'_j$ and $m_j$ agree on $K$ for
sufficiently large $j$.  Either way, $m'_j$ and $z$ will agree on $K$
for sufficiently large $j$ as this is the case for $m_j$.

It remains to show that $\Phi(m'_j)$ does not converge to $\Phi(z)$.
Comparing to the beginning of the proof, the value $\Phi(m'_j)$ is
given by the first case in (\ref{weights}).  Thus, we must show that 
deficiencies $\delta_{[x,m'_j]}(a)$ do not converge to
$\delta_{[x,z]}(a)$.  To see this we note that for each $j$, the
deficiency sets $\mathfrak N_{m_j}(a)$ and $\mathfrak N_{m_j'}(a)$
differ in at exactly one place, either including or deleting $H_j$
from the set. It follows that 
$\delta_{[x,{m_j'}]}(a)=\delta_{[x,m_j]}(a)\pm 1
=\delta_{[x,z]}(a)\pm 1\not=\delta_{[x,z]}(a)$ and the proof is
complete. 
\end{proof}

\begin{figure}[h] 
   \centering
   \includegraphics[width=5in]{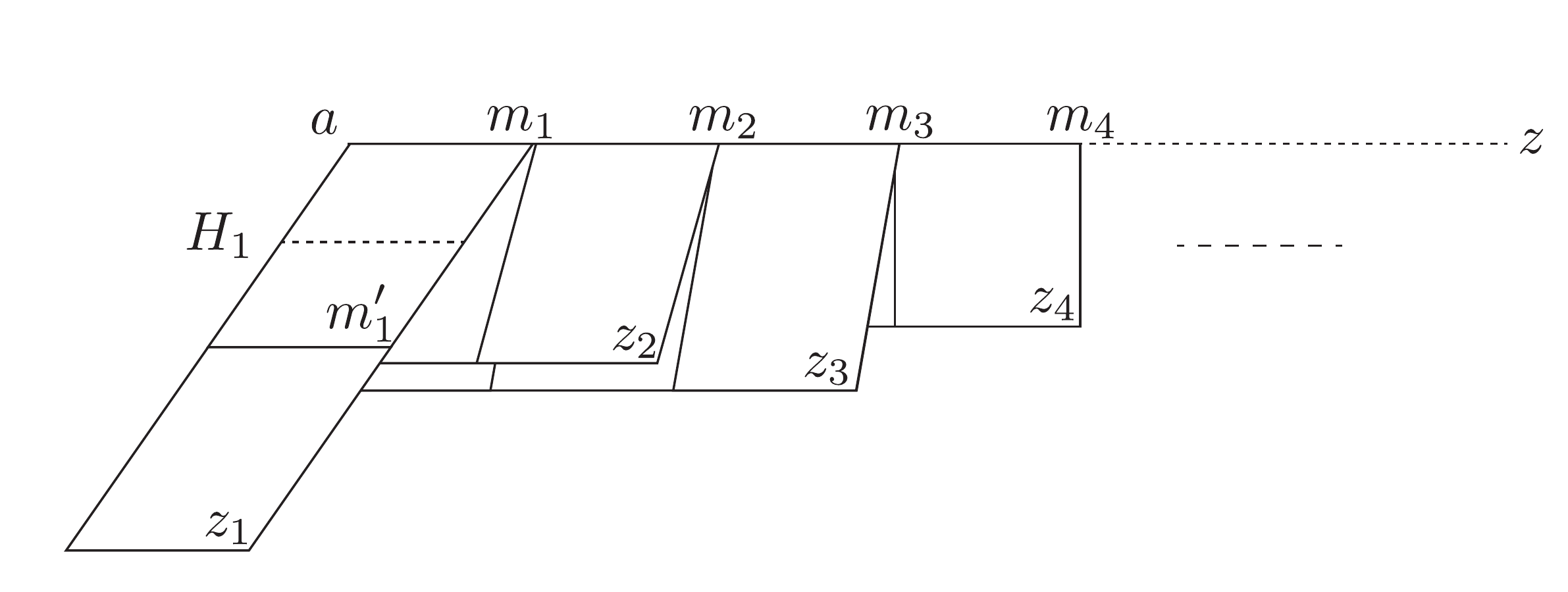} 
   \caption{A point $z$ at which the weight function $\Phi$ is not continuous.}
   \label{fig:example}
\end{figure}

\begin{rem}
  Let $X$ be a (simplicial) tree.  Taken together, the previous
  propositions show that for fixed vertices $x$ and $a\in X$, the function
  $\Phi(z)$ is continuous on all of $\overline X$, except possibly at $a$
  itself.  Further, it is continuous at $a$ exactly when $a$ is finite.
\end{rem}

In summary, when the cube complex is locally finite (that is, every
original vertex is finite) the weight functions are continuous; in
general, however, they are merely Borel.  In either case we shall need
to renormalise to produce probability measures indexed by $\overline X$
while in the latter case we shall also need to replace the Borel
weight functions by a {\it continuous\/} family of probability
measures.  Renormalisation is easy since the weight
functions are all non-negative and have $\ell^1$ norm equal to
${n+N\choose N}$ by Theorem~\ref{oldthm}.  Further, the normalised
weight functions ${n+N\choose N}^{-1} \phi^n_{x,z}$ share the same
continuity and Borel properties as the original $\phi^n_{x,z}$.
Obtaining a continuous family of weight functions is more difficult, but
understood.  The following result is based on the methods of \cite{Brown-Ozawa,ozawa}. 

\begin{lem}\label{Borel2continuous} Let $G$ be a group acting by cellular isometries on a finite dimensional $\CAT(0)$ cube complex $X$.
Given a finite subset $E\subset \Gamma$ and $\varepsilon>0$, there is a
finite subset $F\subset X$ and a function $\eta:\overline
X\rightarrow {\Prob(X)}$ such that  
\begin{nlist}
\item $\eta_z(a)$ is continuous in $z$, for each $a\in {X}$;
  \item $\supp \eta_z\subset F$, for every $z\in {\overline X}$;
  \item $\| s\cdot \eta_z - \eta_{sz} \| < \eps$, for every $s\in E$ and
              every $z\in\overline{X}$.
\end{nlist}
\end{lem}

\begin{proof}[Sketch of proof]
  We sketch the argument of Ozawa, refering to
  \cite[Section~5.2]{Brown-Ozawa} for details.  

  When $X$ is locally finite there is nothing to prove.
  Fixing a basepoint $x\in X$ define $\eta$ using the normalized
  weight functions:  
  $\eta_z(a)= {n+N\choose N}^{-1}\phi^n_{z,x}(a)$. 

  When $X$ is not locally finite the normalised weight functions are
  neither continuous in $z$, nor do they satisfy the conclusion (2) on
  unifom supports.  They are, however, Borel and the proof proceeds by
  applying Lusin's theorem to approximate them by appropriate
  continuous functions $\eta_z$, taking care to ensure that we
  truncate to a common finite subset $F$ throughout. The approximation
  is carried out so that $0\in C(\overline X)$ is in the weak closure
  of the $s\cdot \eta_z - \eta_{sz}$.  Applying the Hahn-Banach
  theorem, after taking convex combinations we obtain (3).

\end{proof}

\begin{rem}
  In fact, we shall not require the full statement of the lemma.  We
  require only the existence, for every finite subset $E$ and
  $\varepsilon>0$, of a finite subset $F$ and functions
  $\eta:X\to\Prob(X)$ satisfying (2) and (3) where in (3) we consider
  only those $z\in X$.
\end{rem}

\section{Permanence}
\label{sec:permanence}

We shall adopt the following characterization of Property $A$ as our
definition.  A countable discrete group $\Gamma$ has 
{\it Property $A$\/} if for every finite subset $E\subset\Gamma$ and
every $\eps>0$ there exists a finite subset $F\subset \Gamma$ and a
function $\nu:\Gamma\to \Prob(\Gamma)$ such that
\begin{nlist}
 \item $\supp \nu_x\subset F$, for every $x\in \Gamma$,
 \item $\| s\cdot \nu_x - \nu_{sx} \| < \eps$, for every $s\in E$ and
             every $x\in\Gamma$.
\end{nlist}
Here, $\Prob(\Gamma)$ is the space of probability measures on $\Gamma$
and the norm is the $\ell_1$-norm.  We refer to
\cite[Lemma~3.5]{higson-roe} for the equivalence with the original
formulation of Property $A$ found in \cite{yu}.  For the present
purposes our definition has two advantages; first it makes no
reference to a particular compact space on which the group acts,  and
second the probability measure associated to a particular $x\in\Gamma$
is supported near the identity of $\Gamma$ and not near $x$ itself.

\begin{thm}\label{permanence}
  Let $\Gamma$ be a countable discrete group acting on a finite
  dimensional cubical complex $X$. Then $\Gamma$ has Property $A$ if
  and only if every vertex stabilizer of the action has Property $A$.
\end{thm}

Since every subgroup of a Property $A$ group has Property $A$ we only
need to prove that if every vertex stabliser has Property $A$ then so
does $\Gamma$. To do so we will have to inflate the Property $A$
functions for the stabilisers to functions defined on the whole group.
After first establishing relevant notation, we shall accomplish this
in the next lemma.

Let $\Delta$ be a subgroup of a group $\Gamma$.  Choose a set $Z$ of
coset representatives for the right cosets of $\Delta$ in $\Gamma$.
Thus, every $g\in \Gamma$ has a unique representation
\begin{equation*}
  g = z_g a_g, \quad\text{$z_g\in Z$, $a_g\in \Delta$}.
\end{equation*}
These satisfy the following properties:
\begin{eqnarray}
   z_{gk} ( a_{gk}a_k^{-1} ) &=& gz_k,
        \quad \text{for $g$, $k\in \Gamma$}\notag \\
   a_{gh} &=& a_g h,  
        \quad \text{for $g\in \Gamma$, $h\in \Delta$} \label{cosetreps} \\
   z_{gh} &=& z_g,    \quad  \text{for $g\in \Gamma$, $h\in \Delta$}; \notag
\end{eqnarray}
indeed, the first follows from $z_{gk}a_{gk} = gk = gz_ka_k$
and the others from $z_{gh}a_{gh}=gh=z_g a_gh$ together with
$a_gh\in \Delta$.

\begin{lem}\label{induction}
  Suppose $\Delta$ is a subgroup of a group $\Gamma$ and that $\Delta$ has Property
  $A$.  For every finite subset $E\subset \Delta$ and every $\eps>0$ there
  exists a finite subset $F\subset \Delta$ and a function 
  $\nu : \Gamma\to\Prob(\Gamma)$ such that
  \begin{nlist}
    \item $\supp(\nu_g)\subset F$, for every $g\in \Gamma$,
    \item $\| h\cdot \nu_g - \nu_{hg} \| < \eps$, for every $h\in E$
      and every $g\in \Gamma$. 
  \end{nlist}
\end{lem}
\begin{proof}
  We shall lift functions obtained from the assumption that $\Delta$ has
  Property $A$ from $\Delta$ to $\Gamma$ using a $\Delta$-equivariant splitting of
  the inclusion $\Delta\subset \Gamma$; we consider $\Delta$ acting on the left of
  both $\Delta$ and $\Gamma$.  Precisely, define
  \begin{equation*}
     \sigma : \Gamma\to \Delta ,\quad \sigma(g) = a_{g^{-1}}^{-1} 
  \end{equation*}
and observe that if $h\in \Delta$ we have
\begin{equation*}
  \sigma(hg) = a_{g^{-1}h^{-1}}^{-1}
             = (a_{g^{-1}}h^{-1})^{-1}
             = h a_{g^{-1}}^{-1} = h \sigma(g),
\end{equation*}
where the second equality follows from (\ref{cosetreps}).  If now $E$ and
$\eps$ are given, we obtain a function $\Delta\to \Prob(\Delta)$ as in the
definition of Property $A$ and define $\nu$ to be the composition
\begin{equation*}
  \Gamma \to \Delta \to \Prob(\Delta) \subset \Prob(\Gamma),
\end{equation*}
in which the first map is our splitting $\sigma$ and we simply view
$\Prob(\Delta)$ as the probability measures on $\Gamma$ which are supported on
$\Delta$.  The required properties are easily verified, with left
$\Delta$-equivariance of $\sigma$ used to verify the norm inequality.
\end{proof}

Now suppose that $\Gamma$ acts on a CAT(0) cube complex by cellular
isometries. As above we obtain an induced continuous action on the
space $\overline X$ of admissible vertices. Fix a transversal $T$ for
the action of $\Gamma$ on $X$; thus, $T\subset X$ contains exactly one
point from each $\Gamma$-orbit. We do not assume that $T$ is
finite. Denote the stabiliser of $t\in T$ by $\Gamma^t$. We apply the
previous notational conventions to $\Gamma^t$. In particular, fixing a
set of coset representatives $Z^t$ for $\Gamma^t$ in $\Gamma$ we have
decompositions
\begin{equation*}
  g= z_g a_g
\end{equation*}
as above, and the previous lemma applies.  As these decompositions
depend on $t\in T$, we should more properly include $t$ in the
notation and write, for example $g=z_g^t a_g^t$.  Observe
\begin{equation*}
  g\cdot t = z_ga_g\cdot t = z_g \cdot t.
\end{equation*}
Thus, the orbit map $g\mapsto g\cdot t$ restricts to a
map $Z^t\to X$, which is a bijection of $Z^t$ onto the
orbit $\Gamma t$ of $t$.

\begin{proof}[Proof of Theorem~\ref{permanence}]
  We are given a finite subset $E\subset \Gamma$ and $\eps>0$.
  Without loss of generality we assume that $E$ is closed under
  inversion and contains the identity of $\Gamma$. We must produce a
  finite subset $F\subset \Gamma$ and a function 
  $\nu:\Gamma \to\Prob(\Gamma)$ as in the definition of Property $A$.

 Applying Lemma \ref{Borel2continuous} (or, more properly, the
 subsequent remark) there is a finite subset
 $F\subset X$ and a function $\eta: X\rightarrow\Prob(X)$
 such that   

\begin{nlist}
  \item $\supp \eta_x\subset F$, for every $x\in X$;
  \item $\| s\cdot \eta_x - \eta_{sx} \| < \eps$, for every $s\in E$ and
              every $x\in X$.
\end{nlist}

Let $T_F\subset T$ be the (finite) set of representatives of those
orbits passing through $F$; in other words, $t\in T_F$ precisely when
$\Gamma \cdot t\cap F$ is nonempty.  For each $t\in T_F$ let $Z^t_F\subset Z^t$
be the (finite) subset of representatives of those $\Gamma^t$ cosets
mapping $t$ into $F$; in other words, $z\in Z^t_F$ precisely when
$z\cdot t\in F$.  
Recall here that the action on $t$ restricted to
coset representatives provides a
bijection of $Z^t$ with the orbit $\Gamma \cdot t$.  
Let $E^t\subset \Gamma^t$ be the (finite) subset
\begin{equation*}
  E^t = \{\, z^{-1}_{sg}s z_g : \text{$s\in E$, $g\in Z^t_F$} \,\}.
\end{equation*}

For each $t\in T_F$, using the hypothesis on $\Gamma^t$ apply
Lemma~\ref{induction} with  
$\Delta= \Gamma^t$ and $E=E^t$ to obtain a finite subset
$F_t\subset\Gamma^t$ and a function $\nu^t:\Gamma\to \Prob(\Gamma)$
such that 
\begin{nlist}
  \item $\supp \nu^t_g \subset F_t$, for every $g\in \Gamma$;
  \item $\| h\cdot \nu^t_g - \nu^t_{hg} \| <\eps$, for every $h\in E^t$
    and $g\in \Gamma$.
\end{nlist}

Define the required function $\mu:\Gamma\to \Prob(\Gamma)$ by choosing
a vertex $O$ as a basepoint and setting, for each $x$ and $g\in \Gamma$,

\begin{equation}
\label{Afcns}
  \mu_x(g) = \sum_{t\in T} 
       \eta_{x\cdot O} (g\cdot t)\, \nu^t_{z^{-1}_gx}(a_g).
\end{equation}
Observe that the sum is actually finite as only finitely many orbits can
cross the (finite) common support $F$ of the $\eta_{x\cdot O}$;
indeed, the sum is over $t\in T_F$.

Let us first address the finiteness of support.  For $\mu_x(g)$ to be
nonzero, there must be $t\in T_F$ for which both factors of the
corresponding summand in (\ref{Afcns}) are nonzero.  Fixing such a $t$
and decomposing $g=z_g a_g$ accordingly we obtain:  
$z_g\cdot t = g\cdot t\in F$ so that $z_g\in Z^t_F$, and also 
$a_g\in F_t$.  It follows that
\begin{equation*}
  \supp \mu_x \subset \bigcup_{t\in T_F} Z^t_F F_t,
\end{equation*}
which is a finite subset of $\Gamma$, not depending on $x$.

Let us next check that each $\mu_x$ is a probability measure.  For
these and other norm estimates below, we shall reindex sums using the
bijection $\Gamma\cong Z^t \Gamma^t$, possible for each fixed $t\in T$.  In
other words, having fixed $t\in T$, we shall replace a sum over
$g\in \Gamma$ by a double sum over $z\in Z^t$ and $g\in z\Gamma^t$ and may
identify the latter as a sum over $\Gamma^t$.
We proceed, recalling that $\eta_{x\cdot O}$ and the 
$\nu^t_{(\cdot)}$ are probability measures, hence $[0,1]$-valued,
\begin{align*}
  \| \mu_x \|_{\ell^1(\Gamma)} &= \sum_{g\in \Gamma} \mu_x(g)
     = \sum_{g\in \Gamma}\sum_{t\in T} \eta_{x\cdot O}(g\cdot t)\, 
                \nu^t_{z^{-1}_gx}(a_g)\\
     &= \sum_{t\in T} \sum_{z\in Z^t} \eta_{x\cdot O} (z\cdot t) 
             \sum_{g\in z\Gamma^t} \nu^t_{z^{-1}x}(a_g) \\
     &= \sum_{t\in T}\sum_{z\in Z^t} \eta_{x\cdot O}(z\cdot t)
      = \sum_{t\in T}\sum_{y\in \Gamma\cdot t} \eta_{x\cdot O}(y) = 1;
\end{align*}
in the second line we use that $g\cdot t=z\cdot t$ for 
$g\in z\Gamma^t$ and that for $g\in\Gamma$ the condition $g\in
z\Gamma^t$ is equivalent to $z_g=z$;
in the third line,  observing that the condition $g\in z\Gamma^t$ is
equivalent to $g=zh$ with $h=a_g$ ranging over the stabiliser $\Gamma^t$,
the sum becomes $\sum_{h\in \Gamma^t} \nu^t_{z^{-1}x}(h) =1$ since
$\nu^t_{z^{-1}x}$ is a probability measure; also
as $z$ ranges over the coset representatives $Z^t$ the value
of $z\cdot t$ ranges over the orbit $\Gamma\cdot t$.

Finally, we check the almost invariance condition.  
We are to estimate
\begin{align*}
  \| s\cdot \mu_x - \mu_{sx} \|_{\ell^1(\Gamma)} &= 
          \sum_{g\in \Gamma}| \mu_{x}(s^{-1}g)-\mu_{sx}(g)| \\
     &\leq \sum_{g\in \Gamma} \sum_{t\in T} 
        \bigl| \eta_{x\cdot O}(z_{s^{-1}g}\cdot t)\ \nu^t_{z^{-1}_{s^{-1}g}x}(a_{s^{-1}g})
             - \eta_{sx\cdot O} (z_g\cdot t)\ \nu^t_{z^{-1}_g sx}(a_g) \bigr|,
\end{align*}
independent of $x\in\Gamma$ and $s\in E$.
We estimate the summand using the triangle inequality
\begin{equation}\label{messy}
    \eta_{x\cdot O}(z_{s^{-1}g}\cdot t)
\left|\nu^t_{z^{-1}_{s^{-1}g}x}(a_{s^{-1}g}) - \nu^t_{z^{-1}_gsx}(a_g)\right|
        +  \left|  \eta_{x\cdot O}(z_{s^{-1}g}\cdot t) - 
           \eta_{sx\cdot O} (z_g\cdot t) \right| \nu^t_{z^{-1}_g sx}(a_g)
\end{equation}
and shall proceed to estimate each term in this expression (or, more accurately,
their sums over $g\in \Gamma$ and $t\in T$).
To estimate the term on the right, observe that for $g$ and $h\in \Gamma$
we have that $z_{gh}\cdot t = gh\cdot t = g z_h\cdot t$
(where all decompositions are with respect to $\Gamma^t$).  Hence, 
fixing $t\in T$ and arguing as above we have
\begin{align*}
  \sum_{g\in \Gamma}
     \left|  \eta_{x\cdot O}(z_{s^{-1}g}\cdot t) - 
           \eta_{sx\cdot O} (z_g\cdot t) \right| 
               \nu^t_{z^{-1}_g sx}(a_g)
   &= \sum_{g\in \Gamma} 
         \left|  \eta_{x\cdot O}(s^{-1} z_{g}\cdot t) - 
           \eta_{sx\cdot O} (z_g\cdot t) \right| \nu^t_{z^{-1}_g sx}(a_g) \\
   &=\sum_{z\in Z^t} \left| \eta_{x\cdot O}(s^{-1} z\cdot t) - 
            \eta_{sx\cdot O} (z\cdot t) \right| 
              \sum_{g\in z\Gamma^t}  \nu^t_{z^{-1} sx}(a_g) \\
   &=\sum_{z\in Z^t} \left| s\cdot \eta_{x\cdot O}(z\cdot t)-
             \eta_{sx\cdot O}(z\cdot t) \right|.
\end{align*}
Taking now the sum
over $t\in T$ and using the assumption that $s\in E$ we have estimated
the right hand term in  (\ref{messy}) by 
\begin{equation*}
  \| s \cdot \eta_{x\cdot O} - \eta_{sx\cdot O} \|_{\ell^1(X)} < \eps.
\end{equation*}
It remains only to estimate the left hand term in (\ref{messy}). 
Again, fix $t\in T$ and reindex the sum over $g\in\Gamma$:
\begin{align}
 \sum_{g\in \Gamma} \eta_{x\cdot O}(z_{s^{-1}g}\cdot t) 
           &\left|\nu^t_{z^{-1}_{s^{-1}g}x}(a_{s^{-1}g}) - 
               \nu^t_{z^{-1}_gsx}(a_g)\right| =  \label{oops} \\
 &=\sum_{z\in Z^t} \sum_{g\in z\Gamma^t}  
   \eta_{x\cdot O}({s^{-1}z}\cdot t)
        \left|\nu^t_{z^{-1}_{s^{-1}z}x}(a_{s^{-1}g}) - 
             \nu^t_{z^{-1}sx}(a_g)\right|; \notag
\end{align}
here, we use the fact that for $g\in z\Gamma^t$ we have
$g\cdot t=z\cdot t$, so that also 
$z_{s^{-1}g}\cdot t = s^{-1}g\cdot t = s^{-1}z\cdot t$.
It follows, in particular, that for $g\in z\Gamma^t$ we have 
$z_{s^{-1}g}=z_{s^{-1}z}$.  Hence, setting
\begin{equation*}
  a = z^{-1}_{s^{-1}g} x = z^{-1}_{s^{-1}z} x, 
     \quad b = z^{-1}_g sx = z^{-1}sx, \quad
        c = ba^{-1} = z^{-1} s z_{s^{-1}z},
\end{equation*}
we see that $a$ and $b$ depend only on $s$, $z$ and $x$, whereas $c$ depends
only on $s$ and $z$.  The calculation
\begin{equation*}
  c\cdot t = z^{-1} s z_{s^{-1}g}\cdot t = 
        z^{-1} s {s^{-1}g} \cdot t = z^{-1}g\cdot t = t.
\end{equation*}
shows that $c\in\Gamma^t$.  Further, we claim that if the summand in
(\ref{oops}) corresponding to a particular $g\in z\Gamma^t$ is nonzero
then $c\in E^t$.   Indeed, if the summand is nonzero then necessarily
$s^{-1}z\cdot t\in F$ or, in other words, $h=z_{s^{-1}z}\in Z^t_F$.  
Now, by evaluating on $t$ we see that $z_{sh}=z$:
\begin{equation*}
  z_{sh}\cdot t = sh\cdot t = s z_{s^{-1}z}\cdot t 
     = s s^{-1}z\cdot t = z \cdot t.
\end{equation*}
Hence, 
\begin{equation*}
  c = z^{-1} s z_{s^{-1}z} = z_{sh}^{-1} s h \in E^t.
\end{equation*}
Putting everything together, using the final small calculation 
$a_g = c a_{s^{-1}g}$, and summing over the nonzero terms in
(\ref{oops}) we obtain
\begin{align*}
   \sum_{z\in Z^t} \eta_{x\cdot O}({s^{-1}z}\cdot t)
      \sum_{g\in z\Gamma^t}  
        \left|\nu^t_{a}(c^{-1}a_{g}) - 
             \nu^t_{ca}(a_g)\right| &=
  \sum_{z\in Z^t} \eta_{x\cdot O}({s^{-1}z}\cdot t)
        \| c\cdot \nu^t_a - \nu^t_{ca} \|_{\ell^1(\Gamma^t)} \\
  &\leq \varepsilon \sum_{z\in Z^t} \eta_{x\cdot O}({s^{-1}z}\cdot t)
\end{align*}
where the estimate comes from the assumptions on $\nu^t$.  Summing
further over $t\in T$, and recalling that $\eta_{x\cdot O}$ is a
probability measure, we have estimated the left hand term in
(\ref{messy}).  
\end{proof}

\begin{rem}
  The formula used to define $\mu$ in the proof reduces to the formula
  used in the previous paper \cite{BCGNW} in the case when the
  stabilizers $\Gamma^t$ are finite, and the functions $\nu^t$ are
  taken to be constant at the uniform probability measure on
  $\Gamma^t$; in other words,
  \begin{equation*}
    \nu^t_g(h) = \begin{cases} |\Gamma^t|^{-1}, &h\in \Gamma^t \\
                               0, &h\notin \Gamma^t. \end{cases}
 \end{equation*}
Such $\nu^t$ satisfies the conclusions of the previous lemma, so is
allowed.
\end{rem}

\begin{rem}
  Ozawa's original treatment constructs a space on which $\Gamma$ will
  act amenably \cite{ozawa}.  We have chosen to avoid the formulation
  in terms of amenable actions because the method seldom produces a
  reasonable space.  It is worth noting, however, that if all
  stabilisers are finite, or even amenable, $\Gamma$ will act amenably
  on $\overline X$.  If, in addition the complex $X$ is locally
  finite, $\Gamma$ will act amenably on the boundary comprised of
  ideal vertices.
\end{rem}

\begin{rem}
  In the locally finite case the result follows from standard
  permanence results, found for example in \cite{kirchberg-wassermann}.
\end{rem}

\section{Artin groups}\label{artin}

A Coxeter matrix is a symmetric matrix $M$, with rows and columns
indexed by a not necessarily finite set $I$, and with matrix elements
$M_{ij}\in \mathbb N\cup \{\infty\}$ satisfying $M_{ii}=1$ for all
$i\in I$. Let $S =\{s_i\mid i\in I\}$ be a set in bijective
correspondence with $I$.  The {\it Coxeter group corresponding to the
  Coxeter matrix $M$\/} is defined by the presentation
\begin{equation*}
  \langle\, S \mid (s_is_j)^{M_{ij}}=1, \forall i,j\in I \,\rangle.
\end{equation*}
The {\it Artin group corresponding to the Coxeter matrix $M$\/} is
defined by the presentation
\begin{equation*}
  \langle\, S \mid  
     (s_is_j)_{M_{ij}}=(s_js_i)_{M_{ij}}, \forall i,j\in I \,\rangle,
\end{equation*}
where $(s_is_j)_{M_{ij}}$ denotes the alternating word
$s_is_js_is_j\ldots s_i$ with ${M_{ij}}$ letters if ${M_{ij}}$ is odd
and the alternating word $s_is_js_is_j\ldots s_is_j$ with ${M_{ij}}$
letters if ${M_{ij}}$ is even.  Considering the equivalent
presentation 
\begin{equation*}
  \langle\, S \mid 
    s_i^2=1, (s_is_j)_{M_{ij}}=(s_js_i)_{M_{ij}}, \forall i,j\in I \,\rangle
\end{equation*}
for the Coxeter group we see that the obvious identification of the
generating sets extends to a surjective homomorphism of the Artin
group onto the Coxeter group with kernel the normal subgroup generated
by the squares of the generators.

For each subset $J\subset I$ denote $S_J=\{s_i, i\in J\}$.  The
subgroup of the Coxeter group generated by $S_J$ is a {\it
  parabolic\/} subgroup.  A parabolic subgroup is a Coxeter group in
its own right -- while not obvious, its presentation is obtained by
deleting from the Coxeter group presentation all generators not in
$S_J$ and all relators involving the deleted generators.  A Coxeter
group, or one of its parabolic subgroups, is {\it spherical\/} if it
is a finite group.

By van der Leck's theorem \cite{Lek} similar statements hold for Artin groups.  The subgroup generated by
$S_J$ is a {\it parabolic\/} subgroup, which is itself an Artin group,
with presentation obtained from the Artin group presentation by
deleting all generators not in $S_J$ and
all relators involving the deleted generators.  An Artin group, or one
of its parabolic subgroups, is of {\it finite type\/} if the
corresponding parabolic subgroup of the Coxeter group is spherical.

A finite type parabolic subgroup of an Artin group is not necessarily
finite. For example if we take the Klein $4$-group, with presentation
\begin{equation*}
  \langle\, s_1, s_2\mid s_1^2,s_2^2,(s_1s_2)^{2} \,\rangle,
\end{equation*}
as our Coxeter group then the associated Artin group has presentation
\begin{equation*}
  \langle\, s_1, s_2 \mid s_1s_2=s_2s_1 \,\rangle.
\end{equation*}
It is free abelian of rank $2$. Since the Klein $4$-group is finite
the entire Artin group is of finite type but clearly not finite.

An Artin group is of {\it type $FC$\/} if the following condition
holds:  whenever $J\subset I$ has the property that the parabolic
subgroups $\langle\, s_i, s_j \,\rangle$ are of finite type for every
pair $i$, $j\in J$ then the parabolic subgroup generated by $S_J$ is
itself of finite type.  Equivalently, given a Coxeter matrix $M$, let
$G$ be the graph with vertex set $I$ and an edge joining $i$ to $j$
whenever the generators $s_i$ and $s_j$ generate a spherical Coxeter
group. The Artin group corresponding to $M$ is of type FC if for every
clique (complete subgraph) in $G$ the corresponding parabolic subgroup
is of finite type.

Charney and Davis have shown that an Artin group can be exhibited as a
complex of groups in which the underlying complex admits a natural
cubical structure \cite{CD}.  Further, they showed that the cube
complex is developable, and is locally $\CAT(0)$ if and only if the
Artin group is of type $FC$.  It follows that when the Artin group is
of type $FC$ the developed cover is a $\CAT(0)$ cube complex on
which the Artin group acts. The vertex stabilisers of this action are,
by construction, the parabolic subgroups of finite type. 
Hence an Artin group of
type $FC$ will act on a finite dimensional $\CAT(0)$ cubical complex with finite type vertex stabilisers. 

Now according to a result of of Cohen and Wales (and, independently,
of Digne), Artin groups of finite type are linear \cite{CW} so that, 
appealing to the theorem of Guentner, Higson and Weinberger,
they are exact \cite{GHW}.  Observing that an Artin group is the
direct union of its finitely generated parabolic subgroups, which are themselves
Artin groups, we obtain as a consequence of Theorem~\ref{permanence}:

\begin{thm}
  An Artin group of type FC is exact. \qed
\end{thm}

\end{document}